\documentclass[10pt]{IEEEtran}

\usepackage{verbatim}
\usepackage{amsfonts,amsmath,mathrsfs,amssymb,amsbsy}
\usepackage[final]{graphicx}
\usepackage{times,cite}
\usepackage{subfig}

\long\def\symbolfootnote[#1]#2{\begingroup%
\def\thefootnote{\fnsymbol{footnote}}\footnote[#1]{#2}\endgroup}

\usepackage{verbatim}
\usepackage[]{float,latexsym}
\usepackage{amsfonts,amsmath,mathrsfs,amssymb,amsbsy}
\usepackage{url}

\newtheorem{theorem}{Theorem}[section]

\newtheorem{algorithm}{Algorithm}[section]

\newtheorem{problem}{Problem}[section]

\newcommand{\Prob}{\mathbb{P}}
\newcommand{\Expect}{\mathbb{E}}
\newcommand{\indic}{\mathbb{I}}

\long\def\symbolfootnote[#1]#2{\begingroup%
\def\thefootnote{\fnsymbol{footnote}}\footnote[#1]{#2}\endgroup}

\DeclareMathOperator*{\esssup}{ess\,sup}

\newcommand{\CADD}{{\mathsf{CADD}}}

\newcommand{\WADD}{{\mathsf{WADD}}}
\newcommand{\FAR}{{\mathsf{FAR}}}

\newcommand{\PDC}{{\mathsf{PDC}}}
\newcommand{\PTC}{{\mathsf{PTC}}}

\newcommand{\tauc}{\tau_{\scriptscriptstyle \mathrm{C}}}

\newcommand{\tauw}{\tau_{\scriptscriptstyle \mathrm{W}}}

\newcommand{\taudeall}{\tau_{\scriptscriptstyle \mathrm{DE-All}}}

\newcommand{\taudcm}{\tau_{\scriptscriptstyle \mathrm{DCM}}}
\newcommand{\taudcs}{\tau_{\scriptscriptstyle \mathrm{DCS}}}

\newcommand{\Piall}{\Pi_{\scriptscriptstyle \mathrm{All}}}

\newcommand{\Pidcm}{\Pi_{\scriptscriptstyle \mathrm{DCM}}}
\newcommand{\Pidcs}{\Pi_{\scriptscriptstyle \mathrm{DCS}}}

\newcommand{\tauwl}{\tau_{{\scriptscriptstyle \mathrm{W}}, \ell}}
\newcommand{\tauwlst}{\tau_{{\scriptscriptstyle \mathrm{W}}, {\ell^*}}}
\newcommand{\taucl}{\tau_{{\scriptscriptstyle \mathrm{C}}, \ell}}

\DeclareMathSizes{3}{3}{3}{3}

\begin{document}
\title{Data-Efficient Quickest Outlying Sequence Detection in Sensor Networks}
 \author{
   \IEEEauthorblockN{
     Taposh Banerjee and Venugopal V. Veeravalli 
     }\\
  \IEEEauthorblockA{
        Department of Electrical \& Computer Engineering, \\
        University of Illinois at Urbana-Champaign, Urbana, USA\\
                              Email: {banerje5,vvv}@illinois.edu 
                    }
    }

\maketitle



\symbolfootnote[0]{\footnotesize
Parts of this paper have been presented at ICASSP 2013. 
 This research was supported in part by the National Science Foundation under 
 grant CCF 08-30169, CCF 11-11342 and DMS 12-22498, through the University of Illinois at 
Urbana-Champaign. This research was also supported in part by the U.S. 
Defense Threat Reduction Agency through subcontract 147755 at the University of Illinois from prime award HDTRA1-10-1-0086.}


\begin{abstract}
A sensor network is considered where at each sensor a sequence of random variables is observed. 
At each time step, 
a processed version of the observations is 
transmitted from the sensors to a common node called the fusion center.
At some unknown point in time the distribution of observations at an unknown subset of the 
sensor nodes changes.
The objective is to detect the outlying sequences as quickly as possible, 
subject to constraints on the 
false alarm rate, the cost of observations taken
at each sensor, and the cost of communication between the sensors and the fusion center. 
Minimax formulations are proposed for the above problem and algorithms
are proposed that are shown to be asymptotically optimal for the proposed 
formulations, as the false alarm rate goes to zero. 
It is also shown, via numerical studies, that the proposed
algorithms perform significantly better than those based on fractional sampling, 
in which the classical algorithms from the literature are used 
and the constraint on the cost of observations is met 
by using the outcome of
a sequence of biased coin tosses, independent of the observation process.
\end{abstract}


\section{Introduction} \label{sec:Intro}
In many engineering applications a sensor network is often deployed to 
observe a phenomenon, and to make 
statistical inference in a distributed and collaborative manner in real time. 
An application of particular interest is the detection of the onset of 
an activity or occurence of an event in/around the object/phenomenon being monitored. 
For example, it is of interest 
to detect the arrival of an animal/bird to its habitat or to detect a sudden 
increase in the stress/strain on the infrastructure being monitored like bridges or 
buildings.  
Other applications include the detection of a bioterrorist attack, the detection of the arrival 
of an intruder in a geographical area, etc. 
The detection of such an event in real time can either be a primary objective 
of the network, or it could be used as a trigger to activate more sophisticated and costly 
monitoring systems or sensors, e.g., a video monitoring system or a human inspection. 
Such detection problems can be modeled under the framework of quickest change detection. 

The problem of quickest change detection (QCD) is well studied in the literature; 
see \cite{poor-hadj-qcd-book-2009}, \cite{tart-niki-bass-2014}, \cite{bass-niki-change-det-book-1993}, 
and \cite{veer-bane-elsevierbook-2013} for a review of QCD. 
In the classical QCD problem there is a single sequence of random variables or vectors. 
Before the change the random vectors have a particular distribution. At some point in time, 
called the change point, the distribution of the vectors changes. The objective is to 
detect this change in distribution as quickly as possible, i.e., with minimum possible delay, 
subject to a constraint on the false alarm rate; see \cite{shew-book-1931}, \cite{girs-rubi-amstat-1952}, 
\cite{page-biometrica-1954}, \cite{shir-siamtpa-1963}, \cite{shir-opt-stop-book-1978}, \cite{lord-amstat-1971}, 
\cite{poll-astat-1985} and \cite{mous-astat-1986}. 
Depending on the availability of the
information on the distribution of the change point,
the QCD problem is either studied in the Bayesian setting of \cite{shir-siamtpa-1963}, 
or in non-Bayesian or minimax settings of \cite{lord-amstat-1971} and \cite{poll-astat-1985}.

In this paper we are interested in the decentralized version of the problem first studied 
in \cite{veer-ieeetit-2001}, and further investigated, for example, in 
 \cite{mei-ieeetit-2005} and \cite{tart-veer-sqa-2008}; 
 see \cite{veer-bane-elsevierbook-2013} and the references therein. 

In the decentralized QCD model there   
is a set of sensors and a central decision maker called the fusion center. 
At each sensor a sequence of random variables is observed over time,  
and at each time step, a processed version
of the observations is transmitted from each sensor to the fusion center.
At the change point the distribution 
of the observations at all the sensor nodes changes. 
The objective in the decentralized model is to find a technique to process the 
observations locally at each sensor, and to find a fusion 
technique to be applied at the fusion center, to detect the change as quickly 
as possible, subject to a constraint on the false alarm rate. 

In modern sensor networks  
there is a cost associated with acquiring observations at each sensor. 
Also, there is a cost associated with the communication between 
the sensors and the fusion center. That is there is a cost associated 
with acquiring data in the system. Thus, the change has to be detected 
in a \textit{data-efficient} way. Moreover, a prior on the change point is often
not available. Also often the change only affects a subset of the sensor nodes, 
and the information on the affected subset and even its size may not be known \textit{a priori}.   
In the QCD literature the latter problem is called a multi-channel QCD problem. 

The QCD problem in a multi-channel setting is studied in \cite{tart-veer-fusion-2002}, \cite{mei-biometrica-2010}
and \cite{mei-isit-2011}. However, in these papers neither the cost of observations at the sensors,
nor the cost of communication between the sensor nodes and the fusion center is taken into account.
In the papers
\cite{veer-ieeetit-2001}, \cite{mei-ieeetit-2005}, \cite{tart-veer-sqa-2008}, and \cite{mei-isit-2011}, 
the cost of communication is controlled by restricting the amount of information transmitted to either one bit or few bits,
but the constraint on the cost of communication is not part of the problem formulation itself.
The QCD problem where the cost of communication is considered explicitly is studied
in \cite{zach-sund-ieeetwc-2008} and \cite{bane-etal-ieeetwc-2011}. However in these papers, the cost
of observations at the sensors is not taken into account. Also the problems are not considered in a multi-channel setting,
i.e., in these papers the change affects all the sensors at the time of change.

In \cite{bane-veer-sqa-2012} and \cite{bane-veer-IT-2013} we extended 
the classical QCD formulations studied in 
\cite{shir-siamtpa-1963}, 
\cite{lord-amstat-1971} and \cite{poll-astat-1985} by putting an additional 
constraint on the cost of observations used in the detection process. 
We proposed problem formulations, for the Bayesian setting in \cite{bane-veer-sqa-2012} 
and for two minimax settings in \cite{bane-veer-IT-2013}, 
where the objective is to minimize some version 
of the average delay, subject to constraints on the false alarm rate and 
a version of the average number of observations taken before the change point.  
For the i.i.d. model we proposed two-threshold extensions of the classical single-threshold   
algorithms, and showed that they are asymptotically optimal for the proposed formulations. 
We also showed via simulations that the two-threshold algorithms we proposed provide 
a significant gain in performance as compared to the approach of fractional sampling, 
in which the constraint on the observation cost is met by skipping samples randomly.  

In \cite{bane-veer-sqa-2014} we extended the results from \cite{bane-veer-IT-2013} 
to sensor networks where the change affects all the sensors. However in the problem formulations 
in \cite{bane-veer-sqa-2014}, the cost of communication between the sensors and the fusion center 
is not taken into account. 

In this paper we extend the results from \cite{bane-veer-IT-2013} 
to a sensor network where the change affects the distribution of observations at an unknown subset of sensors. 
We refer to this problem as the quickest outlying sequence detection problem.   
We propose extensions of the minimax problem formulations from \cite{lord-amstat-1971} and \cite{poll-astat-1985} 
to sensor networks by introducing additional constraints on the cost of observations used 
at each sensors, and the cost of communication between the sensors and the fusion center. 
We propose two algorithms:
the DE-Censor-Max algorithm and the DE-Censor-Sum algorithm. 
Both the algorithms are based on the DE-CuSum algorithm we proposed in \cite{bane-veer-IT-2013}. 
The DE-CuSum algorithm can be used for data-efficient QCD in a single sequence of observations. 
In both the algorithms we propose in this paper, 
the DE-CuSum algorithm is used locally at each sensor; thus ensuring data-efficiency at the sensors.
In both the algorithms the local DE-CuSum statistic is transmitted from the sensors to the fusion center,
if the local DE-CuSum statistic is above a certain threshold; this is censoring.
In the DE-Censor-Max algorithm a change is declared at the fusion center when the \textit{maximum} of the transmitted
DE-CuSum statistics across the streams is above a threshold.
In the DE-Censor-Sum algorithm a change is declared at the fusion center when the \textit{sum} of the transmitted
DE-CuSum statistics across the streams is above a threshold.
We will provide detailed performance analysis of these algorithms. The analysis will reveal that
the DE-Censor-Max algorithm is asymptotically optimal for the problems proposed, when
the change affects \textit{exactly one} stream, as the false alarm rate goes to zero.
Also, using the results in \cite{mei-biometrica-2010}, 
the DE-Censor-Sum algorithm is uniformly asymptotically optimal, 
for each possible post-change scenario, as the false alarm rate goes to zero. 
We will also provide numerical results to compare the performance of the two 
proposed algorithms as a function of the number of outlying streams.

\section{Centralized Minimax Formulations for DE-QCD and the DE-CuSum algorithm}
\label{sec:CENTRAL_FORM}
Since the formulations and algorithm proposed in this paper 
crucially depend on the formulations and the algorithm proposed in \cite{bane-veer-IT-2013},
in this section we provide a detailed overview of the results from \cite{bane-veer-IT-2013}.
In the following, we use $\Prob_n$ to denote the underlying probability measure 
when the change occurs at time $n$, $n\leq \infty$.
We use $\Expect_n$ to denote the corresponding expectation. 
We say $p(\alpha) \sim q(\alpha)$
or $p(\alpha) \leq q(\alpha) (1+o(1))$, as $\alpha \to 0$, to denote
$p(\alpha)/q(\alpha) \to 1$ and $\lim_\alpha p(\alpha)/q(\alpha) \leq 1$,
respectively, as $\alpha \to 0$.
We use $D(f\;||\;g)$ to represent the K-L divergence between the p.d.fs $f$ and $g$.
We assume that the moments of up to third order of all the log likelihood ratios appearing
in this paper are finite and positive.

In \cite{bane-veer-IT-2013} we considered data-efficient quickest change detection 
in a single observation sequence. We considered an observation sequence $\{X_n\}$: $\{X_n\}$ are i.i.d. with
probability density function (p.d.f.) $f_0$
before the change point $\gamma$, and are i.i.d. with p.d.f. $f_1$ after the change point $\gamma$.
A decision maker observes the random variables $\{X_n\}$ over time and has to detect this change 
in distribution as quickly as possible, subject to constraints on the false alarm rate and 
the fraction of time observations are taken before change. We now describe 
the type of policies we consider for this problem.  

Let $S_n$ be the indicator random variable such that $S_n=1$ if $X_n$ is used for decision making,
and $S_n=0$ otherwise. Let
\[
\mathcal{I}_n = \left[ S_1, \ldots, S_n, X_1^{(S_1)}, \ldots, X_n^{(S_n)} \right],
\]
represent the information at time $n$.
Here, $X_i^{(S_i)}$ represents $X_i$ if $S_i=1$, otherwise $X_i$ is absent from the information vector $\mathcal{I}_n$.
Let $\tau$ be a stopping time on the information sequence $\{\mathcal{I}_n\}$, that is,
$\indic_{\{\tau=n\}}$ is a measurable function of $\mathcal{I}_n$. Here, $\indic_F$ represents the indicator of the event $F$.
For time $n\geq1$, based on the information vector $\mathcal{I}_n$, a decision is made whether to \textit{stop and declare change} 
($\tau=n$) or \textit{to continue taking observations} ($\tau > n$). 
If the decision is to \textit{continue}, a decision is made as to whether to \textit{take} or \textit{skip} the observation at time $n+1$. Thus, $S_{n+1}$ is a function of the information available at time $n$, i.e.,
\[
S_{n+1} = \phi_n(\mathcal{I}_n),
\]
where, $\phi_n$ is the control law at time $n$. 
The decision on whether or not to take the first observation is taken without observing $\{X_n\}$. 
In the absence of a prior information on the distribution of $\gamma$, $S_1$ is typically set to $1$, that is
the first observation is always taken.
A policy for data-efficient QCD is 
\[
\Psi = \{\tau, \phi_0, \ldots, \phi_{\tau-1} \}.
\]

To capture the cost of observations used before $\gamma$, we proposed a new metric for data-efficiency in minimax settings,  
the Pre-change Duty Cycle ($\PDC$). Here we consider its variant that we studied in \cite{bane-veer-sqa-2014}: 
\begin{equation}
\label{eq:PDC}
\PDC(\Psi) = \limsup_{\gamma \to \infty}  \frac{1}{\gamma} \Expect_\gamma \left[\sum_{k=1}^{\gamma-1} S_k\right].
\end{equation}
We note that $\PDC\leq 1$. If in a policy all the samples are taken, then the $\PDC$ for that policy is $1$. 
If every other sample is skipped, then the $\PDC$ 
for that policy is $0.5$.

For delay and false alarm we considered the metrics used in \cite{lord-amstat-1971}: the Worst case Average Detection Delay ($\WADD$)
\begin{equation}\label{eq:WADD_Def}
\WADD(\Psi) = \sup_{\gamma\geq 1}  \text{ess sup}\; \Expect_\gamma \left[ (\tau-\gamma)^+ | \mathcal{I}_{\gamma-1} \right],
\end{equation}
and the False Alarm Rate ($\FAR$)
\begin{equation}\label{eq:FAR_Def}
\FAR(\Psi) = 1/\Expect_\infty \left[ \tau\right].
\end{equation}
We considered the following data-efficient minimax formulation. 
\begin{problem}\label{prob:DELorden}
\begin{eqnarray}
\label{eq:MinimaxProblemLorden}
\underset{\Psi}{\text{minimize}} && \WADD(\Psi),\nonumber \\\vspace{-0.15cm}
\text{subject to } && \FAR(\Psi) \leq \alpha, \\
\text{  and  } &&\PDC(\Psi) \leq \beta. \nonumber
\end{eqnarray}
Here $0 \leq \alpha, \beta \leq 1$ are the given constraints.
\end{problem}
We also studied the data-efficient minimax formulation
where instead of $\WADD$, the following Conditional Average Detection Delay ($\CADD$) metric from 
\cite{poll-astat-1985} is used: 
\begin{equation}\label{eq:CADD_Def}
\CADD(\Psi) = \sup_{\gamma \geq 1} \ \ \Expect_\gamma \left[ \tau-\gamma | \tau \geq \gamma \right].
\end{equation}
\begin{problem}\label{prob:DEPollak}
\begin{eqnarray}
\label{eq:MinimaxProblemPollak}
\underset{\Psi}{\text{minimize}} && \CADD(\Psi),\nonumber \\\vspace{-0.15cm}
\text{subject to } && \FAR(\Psi) \leq \alpha, \\
\text{  and  } &&\PDC(\Psi) \leq \beta. \nonumber
\end{eqnarray}
Here $0 \leq \alpha, \beta \leq 1$ are given constraints.
\end{problem}

We then proposed an algorithm, that we called the DE-CuSum algorithm, and showed that it is asymptotically 
optimal for both the above problems, for each fixed $\beta$, as $\alpha\to 0$. The DE-CuSum algorithm is defined below. 
\begin{algorithm}[$\text{DE-CuSum}$]
\label{algo:DECuSum}
Start with $W_0=0$ and fix $\mu>0$, $A>0$ and $h\geq0$. For $n\geq 0$ use the following control:
\begin{enumerate}
\item Take the first observation.
\item If an observaton is taken then update the statistic using 
\[ W_{n+1} = (W_n + \log [f_1(X_{n+1})/f_0(X_{n+1})])^{h+},\]
where $(x)^{h+} = \max\{x, -h\}$.
\item If $W_n < 0$, skip the next observation and update the statistic using 
\[
W_{n+1} = \min\{W_n + \mu, 0\}.
\]
\item Declare change at
\[\tauw = \inf\left\{ n\geq 1: W_n > A \right\}.\]
\end{enumerate}
\end{algorithm}
If $h=0$, the DE-CuSum statistic $W_n$ never becomes negative and hence
reduces to the CuSum statistic \cite{page-biometrica-1954} and evolves as: $C_0=0$, and for $n\geq 0$,\vspace{-.15cm}
\[
C_{n+1} = \max\{0, C_n + \log [f_1(X_{n+1})/f_0(X_{n+1})]\}.
\]
%
The evolution of the DE-CuSum algorithm and the CuSum algorithm for a given sequence of observations is plotted in Fig.~\ref{fig:DECuSum_evolution}.
If $h=\infty$, the DE-CuSum statistic evolves as follows. Initially the DE-CuSum statistic
evolves according to the CuSum statistic till the statistic $W_n$ goes below $0$.
Once the statistic goes below $0$, samples are skipped depending
on the undershoot of $W_n$ (this is also the sum of the log likelihood ratio of the observations) and the design parameter $\mu$.
Specifically, the statistic is incremented by $\mu$ at each time step, and
samples are skipped till $W_n$ goes above zero,
at which time it is reset to zero. At this point, fresh observations are taken and the
process is repeated till the statistic crosses the threshold $A$, at which time a change is declared.
The parameter $\mu$ is a substitute for the Bayesian prior $\rho$
that is used in the DE-Shiryaev algorithm described in \cite{bane-veer-sqa-2012},
and is chosen to meet the constraint on the $\PDC$. 
Thus, the DE-CuSum algorithm is a sequence of SPRTs (\cite{wald-wolf-amstat-1948},
\cite{sieg-seq-anal-book-1985}) intercepted by ``sleep'' times controlled by the undershoot
and the parameter $\mu$. If $h < \infty$, the number of consecutive samples skipped is bounded by $h/\mu + 1$.
\begin{figure}[htb]
\center
\includegraphics[width=9cm, height=5cm]{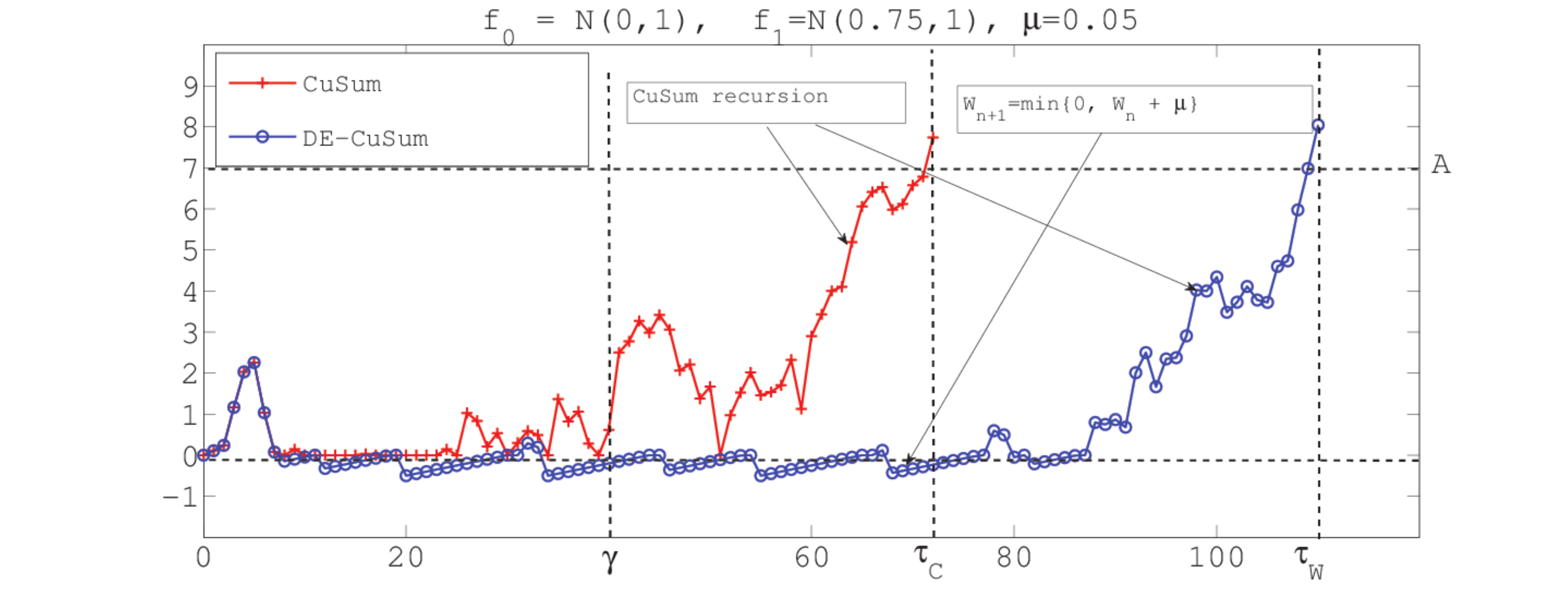}
\caption{{Typical evolution of the CuSum statistic and the DE-CuSum statistic evaluated using 
the same set of observations. Note that the CuSum statistic is always greater than the DE-CuSum statistic. }}
\label{fig:DECuSum_evolution}
\end{figure}

Let $\tauc$ represent the stopping time for the CuSum algorithm. 
It is well known that the CuSum algorithm is asymptotically optimal for both Problem~\ref{prob:DELorden} 
and Problem~\ref{prob:DEPollak}, for $\beta=1$, as $\alpha \to 0$; see \cite{lord-amstat-1971}, \cite{lai-ieeetit-1998} and 
\cite{tart-poll-polu-tpa-2011}. However, since all the observations are taken in the CuSum algorithm 
\[
\PDC(\tauc)=1.
\]
Thus, the CuSum algorithm is not asymptotically optimal for Problem~\ref{prob:DELorden} 
and Problem~\ref{prob:DEPollak}, if $\beta < 1$. 

We proved the following optimality result for the DE-CuSum algorithm 
in \cite{bane-veer-IT-2013}; also see \cite{bane-veer-sqa-2014}. 
Let 
\[
\tau_{-}= \inf\left\{n \geq 1: \sum_{k=1}^n \log \frac{f_1(X_k)}{f_0(X_k)} < 0\right\},
\]
be the ladder variable. Then $W_{\tau_{-}}$ is the ladder height; see \cite{wood-nonlin-ren-th-book-1982} and \cite{sieg-seq-anal-book-1985}. 
Recall that $(x)^{h+} = \max\{x, -h\}$.

\begin{theorem}[\cite{bane-veer-IT-2013}, \cite{bane-veer-sqa-2014}]
\label{thm:DECuSumOpt}
Let $\Expect_1[\log [f_{1}(X_{1})/f_{0}(X_{1})]]$ and $\Expect_\infty[\log [f_{0}(X_{1})/f_{1}(X_{1})]]$ be finite and positive.
If $\mu > 0$, $h < \infty$, and $A=|\log \alpha|$, we have
\begin{equation}
\label{eq:DECuSumPerf}
\begin{split}
C_n &\geq W_n\; \ \ \ \; \forall n \geq 0,\\
\FAR(\tauw) &\leq \FAR(\tauc) \leq \alpha,\\
\PDC(\tauw) &=\frac{\Expect_\infty[\tau_{-}]}{\Expect_\infty[\tau_{-}] + \Expect_\infty[\lceil | W_{\tau_{-}}^{h+}|/\mu \rceil]},\\
\CADD(\tauw) &\leq  \CADD(\tauc) + K,\\
\WADD(\tauw) &\leq \WADD(\tauc) + K.
\end{split}
\end{equation}
In the above equation $K$ is a positive constant that is a function of $\mu$, $h$, and the pre- and post-change
distributions, but is not a function of the threshold $A$. 
If $h=\infty$, then
\begin{equation}
\label{eq:PDCApprox}
\PDC(\tauw) \leq \frac{\mu}{\mu+D( f_0 \; ||\; f_1)}.
\end{equation}
\end{theorem}

Thus, for any fixed threshold, the $\FAR$ of the DE-CuSum algorithm is smaller 
than that of the CuSum algorithm. Thus, we can use the same threshold $A=|\log \alpha|$,
as used for the CuSum algorithm, to satisfy the $\FAR$ constraint of $\alpha$. 
From \cite{lord-amstat-1971} and \cite{lai-ieeetit-1998} it is well known that
\begin{equation}\label{eq:LB_Lai_SingleSensor}
\begin{split}
\inf_{\tau: \FAR(\tau) \leq \alpha} \CADD(\tau) &\sim \CADD(\tauc) \\ 
&\sim \frac{|\log \alpha|}{D(f_1 \; ||\; f_0)}, \mbox{ as } \alpha \to 0. 
\end{split}
\end{equation}
Thus, $\frac{|\log \alpha|}{D(f_1 \; ||\; f_0)}$ is an asymptotic lower 
bound on the $\CADD$ of any stopping time satisfying an $\FAR$ constraint of $\alpha$. 
Because of the above theorem, the $\CADD$ of the DE-CuSum algorithm also achieves this lower bound, 
for each fixed $\mu$ and $h$ (because the delays of the two algorithms are within a constant of each other). 
Also, since 
the expression for the $\PDC$ is not a function of $A$, 
there exists choice of $\mu$ and $h$ such that
the $\PDC$ constraint of $\beta$ can be satisfied independent of the choice of threshold $A$. 
Hence, the DE-CuSum algorithm is asymptotic optimal, 
for both Problem~\ref{prob:DELorden} and Problem~\ref{prob:DEPollak}, for each fixed $\beta$, as $\alpha \to 0$.

To design the DE-CuSum algorithm to achieve a smaller value of $\PDC$, that is to design the 
algorithm to drop a larger fraction of samples before change, one has to select a smaller value for 
the parameter $\mu$. We remark that the assumption that $h< \infty$ is crucial to the proof of the above 
theorem. In \cite{bane-veer-IT-2013} we also showed via simulations that the DE-CuSum algorithm 
provides a significant gain in performance as compared to the approach of fractional sampling, 
where the CuSum algorithm is used and the $\PDC$ constraint is met by skipping samples randomly, 
independent of the observation process.

\section{Problem Formulation for Outlying Sequence Detection in Sensor Networks}
\label{sec:DistForm}
We have a sensor network consisting of $L$ sensors and a central decision
maker called the fusion center.
The sensors are indexed by the index $\ell \in \{1, \cdots, L\}$.
At sensor $\ell$ the sequence $\{X_{n,\ell}\}_{n\geq 1}$ is observed, where $n$ is
the time index.
At $\gamma$, the distribution of $\{X_{n,\ell}\}$ in a subset
$\kappa = \{k_1, k_2, \cdots, k_m\} \subset \{1,2,\cdots,L\}$ of the sensor nodes changes,
from $f_{0,\ell}$ to say $f_{1,\ell}$.
The random variables $\{X_{n,\ell}\}$ are independent across indices $n$ and $\ell$
conditioned on $\gamma$ and the affected subset $\kappa$.
The distributions $f_{0,\ell}$ and $f_{1,\ell}$ are assumed to be known, but
neither the affected subset $\kappa$ nor its size $m$ is known.

Let $S_{n,\ell}$ be the indicator random variable such that
\[
S_{n,\ell} = \begin{cases}
1 & \text{~if~} X_{n,\ell} \text{ is used for decision making at sensor $\ell$}\\
0 & \text{ otherwise}.
\end{cases}
\]

Let $\phi_{n, \ell}$ be the observation control law at sensor $\ell$, i.e.,
\begin{equation*}
S_{n+1,\ell} = \phi_{n, \ell}(\mathcal{I}_{n,\ell}),
\end{equation*}
where
$\mathcal{I}_{n,\ell} = \left[ S_{1,\ell}, \ldots, S_{n,\ell}, X_{1,\ell}^{(S_{1,\ell})}, \ldots, X_{n,\ell}^{(S_{n,\ell})} \right]$.
Here, $X_{n,\ell}^{(S_{n,\ell})} = X_{1,\ell}$ if $S_{1,\ell}=1$, otherwise $X_{1,\ell}$ is absent from the
information vector $\mathcal{I}_{n,\ell} $.
Thus, the decision to take or skip a sample at sensor $\ell$ is based on its past
information.
Let
\[
\boldsymbol{\mathcal{I}}_n=\{\mathcal{I}_{n, 1}, \cdots, \mathcal{I}_{n, L}\}
\]
be the information available at time $n$ across the sensor network.

Also let
\begin{equation*}
Y_{n,\ell} = g_{n,\ell}(\mathcal{I}_{n,\ell})
\end{equation*}
be the information transmitted from sensor $\ell$ to the fusion center. If no information
is transmitted to the fusion center, then $Y_{n,\ell}=\text{NULL}$, which is treated as zero
at the fusion center. Here,
$g_{n,\ell}$ is the transmission control law at sensor $\ell$.
Let
\[\boldsymbol{Y}_n = \{Y_{n,1}, \cdots, Y_{n,L}\}\]
be the information received at the fusion center at time $n$,
and let $\tau$ be a
stopping time on the sequence $\{\boldsymbol{Y}_n\}$.

Let
\[\boldsymbol{\phi_n}=\{\phi_{n, 1}, \cdots, \phi_{n, L}\}\]
denote the observation control law at time $n$, and let
\[\boldsymbol{g}_n=\{g_{n, 1}, \cdots, g_{n, L}\}\]
denote the transmission control law at time $n$.
For data-efficient QCD in sensor networks
we consider the policy of type $\Pi$
defined as
\[\Pi = \{\tau, \{\boldsymbol{\phi}_0, \cdots, \boldsymbol{\phi}_{\tau-1}\}, \{\boldsymbol{g}_1, \cdots, \boldsymbol{g}_\tau\}\}.\]

The $\PDC_\ell$, the $\PDC$ for sensor $\ell$, is defined as
\begin{equation}
\PDC_\ell(\Pi) = \limsup_{\gamma \to \infty}  \frac{1}{\gamma} \Expect_\infty \left[\sum_{k=1}^{\gamma-1} S_{k,\ell}\right].
\end{equation}
Thus, $\PDC_\ell$ is the fraction of time observations are taken before change at sensor $\ell$.

To capture the cost of communication between each sensor and the fusion center before change,
we propose the Pre-Change Transmission Cost ($\PTC$) metric. We
define
\[
T_{n,\ell} = \begin{cases}
1 & \text{~if~} Y_{n,\ell} \neq \text{NULL, i.e, some information }\\
& \text{ is transmitted to the fusion center }\\
0 & \text{ otherwise}.
\end{cases}
\]
The Pre-change Transmission Cost at sensor $\ell$ ($\PTC_\ell$) is defined as
\begin{equation}
\label{eq:PTC_l}
\PTC_\ell(\Pi) = \limsup_{\gamma \to \infty}  \frac{1}{\gamma} \Expect_\infty \left[\sum_{k=1}^{\gamma-1} T_{k,\ell}\right].
\end{equation}
If in a policy every sample is taken and some information is transmitted at every time slot at all
the sensors, then for that policy $\PDC_\ell=\PTC_\ell=1$, $\forall \ell$.
If transmissions happen from the sensors only in every
alternate time slots, then $\PTC_\ell=0.5$, $\forall \ell$.

The objective here is to solve the following extensions of Problem~\ref{prob:DELorden} and
Problem~\ref{prob:DEPollak}: 
\medskip
\begin{problem}\label{prob:MultichannelLorden}\vspace{-0.1cm}
\begin{eqnarray}
\label{eq:MultichannelProblemLorden}
\underset{\Pi}{\text{minimize}} && \WADD(\Pi),\nonumber \\\vspace{-0.1cm}
\text{subject to } && \FAR(\Pi) \leq \alpha, \\
\text{           } &&\PDC_\ell(\Pi) \leq \beta_\ell, \; \mbox{ for } \; \ell=1,\cdots,L, \nonumber\\
\text{  and      } &&\PTC_\ell(\Pi) \leq \sigma_\ell, \; \mbox{ for } \; \ell=1,\cdots,L, \nonumber
\end{eqnarray}
where $0 \leq \alpha, \beta_\ell, \sigma_\ell \leq 1$, for $\ell=1,\cdots,L$, are given constraints, and
\end{problem}
\begin{problem}\label{prob:MultichannelPollak}
\begin{eqnarray}
\label{eq:MultichannelProblemPollak}
\underset{\Pi}{\text{minimize}} && \CADD(\Pi),\nonumber \\\vspace{-0.1cm}
\text{subject to } && \FAR(\Pi) \leq \alpha, \\
\text{           } &&\PDC_\ell(\Pi) \leq \beta_\ell, \; \mbox{ for } \; \ell=1,\cdots,L, \nonumber\\
\text{  and      } &&\PTC_\ell(\Pi) \leq \sigma_\ell, \; \mbox{ for } \; \ell=1,\cdots,L, \nonumber
\end{eqnarray}
where $0 \leq \alpha, \beta_\ell, \sigma_\ell \leq 1$, for $\ell=1,\cdots,L$, are given constraints.
\end{problem}

The asymptotic lower bound developed in \cite{lai-ieeetit-1998} can be specialized to the sensor network setting considered.
Let
\[\Delta_\alpha = \{\Pi: \FAR(\Pi) \leq \alpha\}.\]
\begin{theorem}[\cite{lai-ieeetit-1998}]\label{thm:Multichannel_LB}
If the outlying subset post-change is $\kappa=\{k_1,k_2,\cdots,k_m\}$, then as $\alpha \to 0$,
\begin{equation}\label{eq:MC_LB_CADD}
\inf_{\Pi \in \Delta_\alpha} \CADD(\Pi) \geq \frac{|\log \alpha|}{\sum_{i=1}^m D(f_{1,k_i} \; || \; f_{0,k_i})} (1+o(1)).
\end{equation}
\end{theorem}
\medskip
Since $\WADD(\Pi) \geq \CADD(\Pi)$, we also have as $\alpha \to 0$,
\begin{equation}\label{eq:MC_LB_WADD}
\inf_{\Pi \in \Delta_\alpha} \WADD(\Pi) \geq \frac{|\log \alpha|}{\sum_{i=1}^m D(f_{1,k_i} \; || \; f_{0,k_i})} (1+o(1)).
\end{equation}

\section{Data-Efficient Algorithms for Outlying Sequence Detection in Sensor Networks}
\label{sec:Algorithms}
In this section we propose two algorithms that can be used to detect the outlying sequences in a data-efficient
way in a sensor network. In both the algorithms the DE-CuSum algorithm (Algorithm~\ref{algo:DECuSum})
is used locally at each sensor. In the rest of the paper we use $W_{n,\ell}$ to denote the 
DE-CuSum statistic at sensor $\ell$. 

\subsection{The DE-Censor-Max Algorithm}
\label{sec:DECENSORMAX_algo}
In the DE-Censor-Max algorithm,
the DE-CuSum algorithm is used at each sensor $\ell$.
If the DE-CuSum statistic $W_{n,\ell}$ at a sensor is above a threshold $D_\ell$,
then the statistic is transmitted to the fusion
center. A change is declared at the fusion
center, if the \text{maximum} of the transmitted statistics from all the sensors is
larger than another
threshold $A$.
Mathematically, the DE-Censor-Max algorithm is described as follows.

\begin{algorithm}[$\text{DE-Censor-Max}$: $\Pidcm$]
\label{algo:DE-MAX}
Start with $W_{0,\ell}=0$, $\forall \ell$. Fix $\mu_\ell > 0$, $h_\ell\geq0$, $D_\ell \geq 0$ and  $A \geq 0$.
For $n\geq 0$ use the following control:
\begin{enumerate}
\item Use the DE-CuSum algorithm at each sensor $\ell$, i.e., update the statistics $\{W_{n,\ell}\}_{\ell=1}^L$
for $n\geq 1$ using
\begin{equation*}
\begin{split}
S_{n+1, \ell} &= 1 \text{~only if~} W_{n,\ell} \geq 0 \\
W_{n+1,\ell} &=
\begin{cases}
\min\{W_{n,\ell} + \mu_\ell, 0\} & \hspace{-0.3cm}\text{ if } S_{n+1,\ell} = 0\\
\left(W_{n,\ell} +\log \frac{f_{1,\ell}(X_{n+1,\ell})}{f_{0,\ell}(X_{n+1,\ell})}\right)^{h+} & \hspace{-0.3cm}\text{ if } S_{n+1,\ell} = 1,
\end{cases}
\end{split}
\end{equation*}
where $(x)^{h+} = \max\{x, -h\}$.
\item Transmit
\[
Y_{n,\ell} = W_{n,\ell} \indic_{\{W_{n,\ell} > D_\ell\}}.
\]
\item At the fusion center stop at
\[
\taudcm = \inf\{n \geq 1: \max_{ \ell \in \{1, \cdots, L\}} Y_{n,\ell}> A\}.
\]
\end{enumerate}
\end{algorithm}

With $D_\ell=0$ and $h_\ell=0$, $\forall \ell$, the DE-CuSum algorithm
at each sensor reduces to the $\mathrm{CuSum}$
algorithm, and $Y_{n,\ell} = W_{n,\ell}$ $\forall n,\ell$.
In this case, the DE-Censor-Max algorithm reduces to the MAX
algorithm proposed in \cite{tart-veer-fusion-2002}.

We will show in the Section~\ref{sec:DEMAX_opt} that when exactly one of the $L$ sensor is affected post-change,
then this algorithm is uniformly asymptotically optimal for both Problem~\ref{prob:MultichannelLorden}
and Problem~\ref{prob:MultichannelPollak} (achieves the lower bound provided in Theorem~\ref{thm:Multichannel_LB} for
each $\kappa$),
for each fixed $\{\beta_\ell\}$ and $\{\sigma_\ell\}$, as $\alpha \to 0$.

\subsection{The DE-Censor-Sum Algorithm}
\label{sec:DECENSORSUM_algo}
Although the DE-Censor-Max algorithm is asymptotically optimal, we will show in Section~\ref{sec:DEMultiNumerical}
that it performs poorly when the size of the outlying subset is large. To address this deficiency, we
propose the DE-Censor-Sum algorithm.

In the DE-Censor-Sum algorithm,
the DE-CuSum algorithm is used at each sensor.
If the DE-CuSum statistic at a sensor is above a threshold,
then the statistic is transmitted to the fusion
center. A change is declared at the fusion
center, if the \text{sum} of the transmitted statistics from all the sensors is
larger than another
threshold. 
Mathematically, the DE-Censor-Sum algorithm is described as follows.

\begin{algorithm}[$\text{DE-Censor-Sum}$: $\Pidcs$]
\label{algo:DE-SUM}
Start with $W_{0,\ell}=0$, $\forall \ell$. Fix $\mu_\ell > 0$, $h_\ell\geq0$, $D_\ell \geq 0$ and  $A \geq 0$.
For $n\geq 0$ use the following control:
\begin{enumerate}
\item Use the DE-CuSum algorithm at each sensor $\ell$, i.e., update the statistics $\{W_{n,\ell}\}_{\ell=1}^L$
for $n\geq 1$ using
\begin{equation*}
\begin{split}
S_{n+1, \ell} &= 1 \text{~only if~} W_{n,\ell} \geq 0 \\
W_{n+1,\ell} &=
\begin{cases}
\min\{W_{n,\ell} + \mu_\ell, 0\} & \hspace{-0.3cm}\text{ if } S_{n+1,\ell} = 0\\
\left(W_{n,\ell} +\log \frac{f_{1,\ell}(X_{n+1,\ell})}{f_{0,\ell}(X_{n+1,\ell})}\right)^{h+} & \hspace{-0.3cm}\text{ if } S_{n+1,\ell} = 1,
\end{cases}
\end{split}
\end{equation*}
where $(x)^{h+} = \max\{x, -h\}$.
\item Transmit
\[
Y_{n,\ell} = W_{n,\ell} \indic_{\{W_{n,\ell} > D_\ell\}}.
\]
\item At the fusion center stop at
\[
\taudcs = \inf\{n \geq 1: \sum_{ \ell \in \{1, \cdots, L\}} Y_{n,\ell}> A\}.
\]
\end{enumerate}
\end{algorithm}

With $D_\ell=0$ and $h_\ell=0$, $\forall \ell$, the DE-CuSum algorithm
at each sensor reduces to the CuSum
algorithm, and $Y_{n,\ell} = W_{n,\ell}$ $\forall n,\ell$.
In this case, the DE-Censor-Sum algorithm reduces to the $N_{sum}$
algorithm proposed in \cite{mei-biometrica-2010}.
If $h_\ell=0$ $\forall \ell$ and $D>0$,
the DE-Censor-Sum algorithm reduces to the $N_{hard}$
algorithm proposed in \cite{mei-isit-2011}.
The DE-Censor-Sum algorithm can easily be modified to obtain data-efficient
extensions of other algorithms proposed in \cite{mei-isit-2011}.

We will provide a detailed performance analysis of the DE-Censor-Sum algorithm using which
the threshold $A$ and the parameter $h_\ell$ and $\mu_\ell$ can be selected.
We will use the performance analysis to show that,
under the additional assumption a result in \cite{mei-biometrica-2010}, 
the DE-Censor-Sum algorithm is also uniformly asymptotically optimal
for both Problem~\ref{prob:MultichannelLorden}
and Problem~\ref{prob:MultichannelPollak} (achieves the lower bound provided in Theorem~\ref{thm:Multichannel_LB} for
each $\kappa$),
for each fixed $\{\beta_\ell\}$ and $\{\sigma_\ell\}$, as $\alpha \to 0$.

\section{Asymptotic Optimality of the DE-Censor-Max Algorithm}
\label{sec:DEMAX_opt}
In this section we now show the asymptotic optimality of the DE-Censor-Max algorithm. 

We define the ladder variable \cite{wood-nonlin-ren-th-book-1982}
corresponding to sensor $\ell$:
\[
\tau_{\ell-} = \inf\left\{n \geq 1: \sum_{k=1}^n \log \frac{f_{1,\ell}(X_{k,\ell})}{f_{0,\ell}(X_{k,\ell})} < 0\right\},
\]
and note that $W_{\tau_{\ell-}}$ is the ladder height.
Also, let
\begin{equation*}
U_{D_\ell} = \hspace{-0.1cm}\left\{ \#\mbox{ times}: \sum_{k=1}^n  \log \frac{f_{1,\ell}(X_{k,\ell})}{f_{0,\ell}(X_{k,\ell})}>D_\ell \mbox{ before it's} < 0 \right\}.
\end{equation*}
Thus, $U_{D_\ell}$ is the number of times the random walk $\sum_{k=1}^n \log \frac{f_{1,\ell}(X_{k,\ell})}{f_{0,\ell}(X_{k,\ell})} $
is above $D_\ell$ before hitting $0$. We note that $U_{D_\ell}=0$ with a positive probability.
We note that $U_{D_\ell}$ is also the number of times the DE-CuSum statistic $W_{n,\ell}$ is above $D_\ell$ before hitting $0$.

Also for $x_\ell$ real define
\begin{equation}\label{eq:tauwl_def}
\tauwl(x_\ell, A) = \inf \{n \geq 1: W_{n,\ell} > A; \mbox{ with } W_{0,\ell}=x_\ell\}.
\end{equation}
Thus, $\tauwl(x_\ell, A)$ is the time for the DE-CuSum statistic at sensor $\ell$ to cross the threshold $A$ 
starting with the initial value $x_\ell$. It follows from Lemma 5 in \cite{bane-veer-IT-2013} that
\begin{equation}\label{eq:resettingLemma}
\Expect_1[\tauwl(x_\ell, A)] \leq \Expect_1[\tauwl(0, A)] + \lceil h_\ell/\mu_\ell\rceil, \; \forall \ell.
\end{equation}

\begin{theorem}\label{thm:DEMAX}
Let
\[
 0< D(f_{1,\ell}\; ||\; f_{0,\ell}) < \infty \;\mbox{ and } \; 0< D(f_{0,\ell}\; ||\; f_{1,\ell}) < \infty \quad \forall \ell.
\]
Let $\mu_\ell > 0$, $h_\ell < \infty$, $\forall \ell$, $D_\ell \geq 0$, and $A=\log \frac{L}{\alpha}$.
If the change occurs in the stream $\ell^*$, then we have
\begin{equation}
\label{eq:DEMaxPerf}
\begin{split}
\FAR(\Pidcm) &\leq \alpha, \\
\PDC_\ell(\Pidcm) &=\frac{\Expect_\infty[\tau_{\ell-}]}{\Expect_\infty[\tau_{\ell-}] +
\Expect_\infty[\lceil | W_{\tau_{\ell-}}^{h_\ell+}|/\mu_\ell \rceil]}, \;\forall \ell,\\
\PTC_\ell(\Pidcm) &=\frac{\Expect_\infty[U_{D_\ell}]}{\Expect_\infty[\tau_{\ell-}] +
\Expect_\infty[\lceil | W_{\tau_{\ell-}}^{h_\ell+}|/\mu_\ell \rceil]}, \;\forall \ell,\\
\WADD(\Pidcm) &\leq \frac{|\log \alpha|}{D(f_{1,\ell^*} \; || \; f_{0,\ell^*})}  (1+o(1)) \mbox{ as } \alpha \to 0.
\end{split}
\end{equation}
If $h_\ell=\infty$, $\forall \ell$, then
\begin{equation}
\label{eq:PDCApprox3}
\PDC_\ell(\Pidcm) \leq \frac{\mu_\ell}{\mu_\ell+D( f_{0,\ell} \; ||\; f_{1,\ell})}\; \forall \ell.
\end{equation}
\end{theorem}
\begin{proof}
The $\FAR$ result follows from Theorem 1 of \cite{tart-veer-fusion-2002} because the DE-CuSum 
statistic at each sensor is always greater than the CuSum statistic computed using the same set 
of observations; see Theorem~\ref{thm:DECuSumOpt}. 

The results on $\PDC_\ell$ follows from Theorem~\ref{thm:DECuSumOpt}.
The idea behind the proof is to define an on-off renewal process, with 
on times distributed according to the time for which the DE-CuSum statistic $W_{n,\ell}$ at sensor $\ell$ is above $0$, 
and the off times distributed according to the sojourn of the DE-CuSum statistic $W_{n,\ell}$ at sensor $\ell$ is above $0$. 
The result then follows from the renewal reward theorem. 
The results on $\PTC_\ell$ follows also from the renewal reward theorem and the arguments are almost identical to
those provide for $\PDC_\ell$.

The delay result holds because after change the max of statistics is greater than the statistics in which the change has
taken place. Thus, the delay of the DE-Censor-Max algorithm is bounded from above by the delay of the
DE-CuSum algorithm when applied to the outlying sensor. Mathematically, the argument is as follows.

We obtain an upper bound on $\Expect_\gamma \left[ (\taudcm-\gamma)^+ | \boldsymbol{\mathcal{I}}_{\gamma-1} \right]$
that is not a function of $\gamma$ and the conditioning $\boldsymbol{\mathcal{I}}_{\gamma-1}$, and that scales as the lower bound in Theorem~\ref{thm:Multichannel_LB}. The theorem is then established if we then take the essential supremum and then
the supremum over $\gamma$.

Let $\boldsymbol{\mathcal{I}}_{\gamma-1}=\boldsymbol{i}_{\gamma-1}$ be such that
$W_{\gamma-1, \ell}=x_\ell$, $x_\ell \in [-h_\ell, \infty)$.
We first note that for $A > \max_\ell D_\ell$,
\begin{equation}
\begin{split}
\Expect_\gamma \left[ (\taudcm-\gamma)^+ |  \boldsymbol{\mathcal{I}}_{\gamma-1}=\boldsymbol{i}_{\gamma-1} \right]
\leq \Expect_1\left[\tauwlst(x_{\ell^*}, A)\right],
\end{split}
\end{equation}
where $\tauwl(x_\ell,A)$ is as defined in \eqref{eq:tauwl_def}. Then from \eqref{eq:resettingLemma} we have
\begin{equation}
\begin{split}
\Expect_\gamma &\left[ (\taudcm-\gamma)^+  |  \boldsymbol{\mathcal{I}}_{\gamma-1}  = \boldsymbol{i}_{\gamma-1} \right]\\
& \leq \Expect_1\left[\tauwlst(x_{\ell^*},A)\right] \leq \Expect_1[\tauwlst(0,A)] + \lceil h_{\ell^*}/\mu_{\ell^*} \rceil.
\end{split}
\end{equation}
The result now follows from the proof of Theorem~\ref{thm:DECuSumOpt} on the DE-CuSum algorithm because the change 
is assumed to affect the sensor $\ell^*$ and $h_{\ell^*} < \infty$. 
\end{proof}
Since $\CADD \leq \WADD$, we also have under the same assumptions as in Theorem~\ref{thm:DEMAX}
\begin{equation}
\begin{split}
\CADD(\Pidcm) \leq \frac{|\log \alpha|}{D(f_{1,\ell^*} \; || \; f_{0,\ell^*})}   (1+o(1)), \mbox{ as } \alpha \to 0.
\end{split}
\end{equation}

From Theorem~\ref{thm:Multichannel_LB}, the $\WADD$, and hence the $\CADD$ performance of the
DE-Censor-Max algorithm is the best one can do when the change affects the stream $\ell^*$,
for given $\{\beta_\ell\}$ and $\{\sigma_\ell\}$, as $\alpha \to 0$.
Also, the $\PDC_\ell$ and the $\PTC_\ell$ performances do not depend on the threshold $A$, thus
the constraints $\{\beta_\ell\}$ and $\{\sigma_\ell\}$ can be satisfied independent of the $\FAR$ constraint $\alpha$.
Hence, the DE-Censor-Max algorithm is asymptotically optimal when the change affects exactly one stream,
for both Problem~\ref{prob:MultichannelLorden}
and Problem~\ref{prob:MultichannelPollak}, for each given
$\{\beta_\ell\}$ and $\{\sigma_\ell\}$, as $\alpha \to 0$.

\subsection{Performance Analysis of the DE-Censor-Sum Algorithm}
\label{sec:DECENSORSUM_opt}

In this section we provide the performance analysis of the DE-Censor-Sum algorithm and comment 
on its asymptotic optimality. In Theorem~\ref{thm:DECensorSum} to be proved below, the delay 
analysis of the DE-Censor-Sum depends on the delay analysis of the DE-All algorithm that we studied 
in \cite{bane-veer-sqa-2014}. We first define the DE-All algorithm. 

In the DE-All algorithm, the DE-CuSum algorithm (see Algorithm~\ref{algo:DECuSum})
is used at each sensor,
 and a ``1'' is transmitted
each time the DE-CuSum statistic at any sensor is above a threshold. A change is declared the first
time a ``1'' is received at the fusion center
 from \textit{all} the sensors at the same time.

Let
\[d_\ell = \frac{D(f_{1,\ell} \; || \; f_{0,\ell})}{\sum_{k=1}^L D(f_{1,k} \; || \; f_{0,k})}.\]
\begin{algorithm}[$\mathrm{DE-All}$: $\Piall$]
\label{algo:DE-All}
Start with $W_{0,\ell}=0$ $\forall \ell$. Fix $\mu_\ell > 0$, $h_\ell\geq0$, and  $A \geq 0$.
For $n\geq 0$ use the following control:
\begin{enumerate}
\item Use the DE-CuSum algorithm at each sensor $\ell$, i.e., update the statistics $\{W_{n,\ell}\}_{\ell=1}^L$
for $n\geq 1$ using
\begin{equation*}
\begin{split}
S_{n+1, \ell} &= 1 \text{~only if~} W_{n,\ell} \geq 0 \\
W_{n+1,\ell} &=
\begin{cases}
\min\{W_{n,\ell} + \mu_\ell, 0\} & \hspace{-0.3cm}\text{ if } S_{n+1,\ell} = 0\\
\left(W_{n,\ell} +\log \frac{f_{1,\ell}(X_{n+1,\ell})}{f_{0,\ell}(X_{n+1,\ell})}\right)^{h+} & \hspace{-0.3cm}\text{ if } S_{n+1,\ell} = 1,
\end{cases}
\end{split}
\end{equation*}
where $(x)^{h+} = \max\{x, -h\}$.
\item Transmit
\[
Y_{n,\ell} = \indic_{\{W_{n,\ell} > d_\ell A\}}.
\]
\item At the fusion center stop at
\[
\taudeall = \inf\{n \geq 1: Y_{n,\ell}=1 \mbox{ for all } \ell \in \{1, \cdots, L\}\}.
\]
\end{enumerate}
\end{algorithm}

We now provide the main result of this section. For that we define another variable:
\begin{equation}\label{eq:tauCl_def}
\taucl(x_\ell, A) = \inf \{n \geq 1: C_{n,\ell} > A; \mbox{ with } C_{0,\ell}=x_\ell\}.
\end{equation}
Thus, $\taucl(x_\ell, A)$ is the time for the CuSum statistic at sensor $\ell$ $C_{n,\ell}$, to cross the threshold $A$ 
starting with the initial value $x_\ell$. 

\begin{theorem}\label{thm:DECensorSum}
Let
\[
 0< D(f_{1,\ell}\; ||\; f_{0,\ell}) < \infty \; \mbox{ and } \; 0< D(f_{0,\ell}\; ||\; f_{1,\ell}) < \infty \quad \forall \ell.
\]
Let $\mu_\ell > 0$, $h_\ell < \infty$, $\forall \ell$, $D_\ell \geq 0$, and $A=L \log \frac{L}{\alpha}$.
If the change affects the subset $\kappa$ of streams, then we have
\begin{equation}
\label{eq:DESumPerf}
\begin{split}
\FAR(\Pidcs) &\leq \alpha, \\
\PDC_\ell(\Pidcs) &=\frac{\Expect_\infty[\tau_{\ell-}]}{\Expect_\infty[\tau_{\ell-}] +
\Expect_\infty[\lceil | W_{\tau_{\ell-}}^{h_\ell+}|/\mu_\ell \rceil]}, \; \forall \ell\\
\PTC_\ell(\Pidcs) &=\frac{\Expect_\infty[U_{D_\ell}]}{\Expect_\infty[\tau_{\ell-}] +
\Expect_\infty[\lceil | W_{\tau_{\ell-}}^{h_\ell+}|/\mu_\ell \rceil]},\; \forall \ell\\
\WADD ( \Pidcs) &\\
& \hspace{-0.5cm}\leq \frac{A}{\sum_{i=1}^m D(f_{1,k_i} \; || \; f_{0,k_i})}  (1+o(1)) \mbox{ as } A \to \infty.
\end{split}
\end{equation}
If $h_\ell=\infty$, $\forall \ell$, then
\begin{equation}
\PDC_\ell(\Pidcs) \leq \frac{\mu_\ell}{\mu_\ell+D( f_{0,\ell} \; ||\; f_{1,\ell})}\; \forall \ell.
\end{equation}
\end{theorem}
\begin{proof}
The proofs on $\PDC_\ell$ and $\PTC_\ell$ are identical to that provided in the Theorem~\ref{thm:DEMAX}.

For the $\FAR$ note that
\[
\left\{\sum_{\ell=1}^L W_{n,\ell} > A\right\} \subset \left\{\max_{\ell\in \{1,\cdots,L\}} W_{n,\ell} > \frac{A}{L}\right\}.
\]
For simplicity we write $\Pidcs(A)$ to represent DE-Censor-Sum algorithm when the threshold used at the fusion center is $A$.
Similarly we use $\Pidcm(A/L)$ to represent DE-Censor-Max algorithm when the threshold used at the fusion center is $A/L$.
Then the above subset relation implies
\[
\FAR(\Pidcs(A)) \leq \FAR(\Pidcm(A/L)).
\]
The $\FAR$ result follows because from Theorem~\ref{thm:DEMAX}
we have that
\[
\FAR(\Pidcm(A/L)) \leq \alpha \quad \mbox{ if } \quad A/L = \log L/\alpha.
\]

For the $\WADD$ analysis, let $\taudcs(\kappa)$ denote the DE-Censor-Sum algorithm
applied to only the streams in the affected subset $\kappa$. Further let
$\boldsymbol{\mathcal{I}}_{\gamma-1}(\kappa)$ denote the information in the affected streams.
Then
\begin{equation}\label{eq:DESumProof_1}
\begin{split}
\Expect_\gamma \left[ (\taudcs-\gamma)^+ |  \boldsymbol{\mathcal{I}}_{\gamma-1}=\boldsymbol{i}_{\gamma-1} \right] &\\
&\hspace{-2.5cm}\leq \Expect_\gamma \left[ (\taudcs(\kappa)-\gamma)^+ |  \boldsymbol{\mathcal{I}}_{\gamma-1}=\boldsymbol{i}_{\gamma-1} \right] \\
&\hspace{-2.5cm}=\Expect_\gamma \left[ (\taudcs(\kappa)-\gamma)^+ |  \boldsymbol{\mathcal{I}}_{\gamma-1}(\kappa)=\boldsymbol{i}_{\gamma-1}(\kappa) \right].
\end{split}
\end{equation}
In the above equation the last equality is true because the observations across the streams are independent conditioned 
on the change point. 
Because of the above inequality, we can assume that the change affects all the subsets at the same time, i.e.,
$\kappa = \{1,\cdots, L\}$.

Now note that any fixed $A$ (see Algorithm~\ref{algo:DE-All})
\[
\left\{ W_{n,\ell} > d_\ell A, \; \forall \ell\right\} \subset \left\{ \sum_\ell W_{n,\ell} > \sum_\ell d_\ell A = A \right\}.
\]
Hence, for $A$ sufficiently large and from the proof of Theorem 6.1 in \cite{bane-veer-sqa-2014}, we have
\begin{equation}
\begin{split}
\Expect_\gamma \left[ (\taudcs-\gamma)^+ |  \boldsymbol{\mathcal{I}}_{\gamma-1}=\boldsymbol{i}_{\gamma-1} \right] &\\
&\hspace{-2.5cm} \leq \Expect_\gamma \left[ (\taudeall-\gamma)^+ |  \boldsymbol{\mathcal{I}}_{\gamma-1}=\boldsymbol{i}_{\gamma-1} \right]\\
&\hspace{-2.5cm} \leq \Expect_1\left[\max_{1\leq \ell \leq L} \taucl \right] + \mbox{ constant}.
\end{split}
\end{equation}
The proof of the theorem is now complete because we can now take $\esssup$ and then $\sup$ over
$\gamma$ on the left-hand side. Then, from \cite{mei-ieeetit-2005} it follows that $\Expect_1\left[\max_{1\leq \ell \leq L} \taucl \right]$ grows
in the order $\frac{A}{\sum_{\ell=1}^L D(f_{1,\ell} \; || \; f_{0,\ell})}$ . This when applied to the affected subset $\kappa$
gives us the desired result on the $\WADD$ from \eqref{eq:DESumProof_1}.
\end{proof}

Note that the above theorem does not imply the asymptotic optimality of the DE-Censor-Sum algorithm, mainly
due to the fact that the choice of the threshold is conservative. It only gives a delay bound of $L$ times that of the 
lower bound in Theorem~\ref{thm:Multichannel_LB}. 
However, if the threshold can be set to be of the order $\log 1/\alpha$ to satisfy the $\FAR$ constraint,
then the above theorem establishes the uniform asymptotic optimality of the DE-Censor-Sum algorithm for each possible post-change
distribution. It is shown in \cite{mei-biometrica-2010} that such a result is indeed true, 
and therefore 
under the conditions of Theorem~\ref{thm:DECensorSum} above, 
the DE-Censor-Sum algorithm is uniformly asymptotically 
optimal, for each possible $\kappa$, for each fixed $\{\beta_\ell\}$ and $\{\sigma_\ell\}$, as $\alpha \to 0$. 
\medskip

\section{Numerical Results}
\label{sec:DEMultiNumerical}
We first compare the performance of the DE-Censor-Sum algorithm, the DE-Censor-Max algorithm and the
Oracle CuSum algorithm as a function of the number of affected stream. The Oracle
CuSum algorithm uses all of the sensor observations, and it knows 
the indices of the outlying sequences.
We plot the $\CADD$ versus the number of affected stream comparison in Fig.~\ref{fig:MAX_SUM_Comp} for the parameters:
$\FAR=10^{-3}$, $L=100$, $f_{0,\ell}=f_0=\mathcal{N}(0,1)$, $\forall \ell$, $f_{1,\ell}=f_1=\mathcal{N}(0.5,1)$, $\forall \ell$, and
for the $\{\PDC_\ell\}$ and $\{\PTC_\ell\}$ constraints of $\beta_\ell=\sigma_\ell=0.5$, $\forall \ell$. We also set the local
thresholds $D_\ell=0$, $\forall \ell$.
\begin{figure}[htb]
\center
\includegraphics[width=9cm, height=5cm]{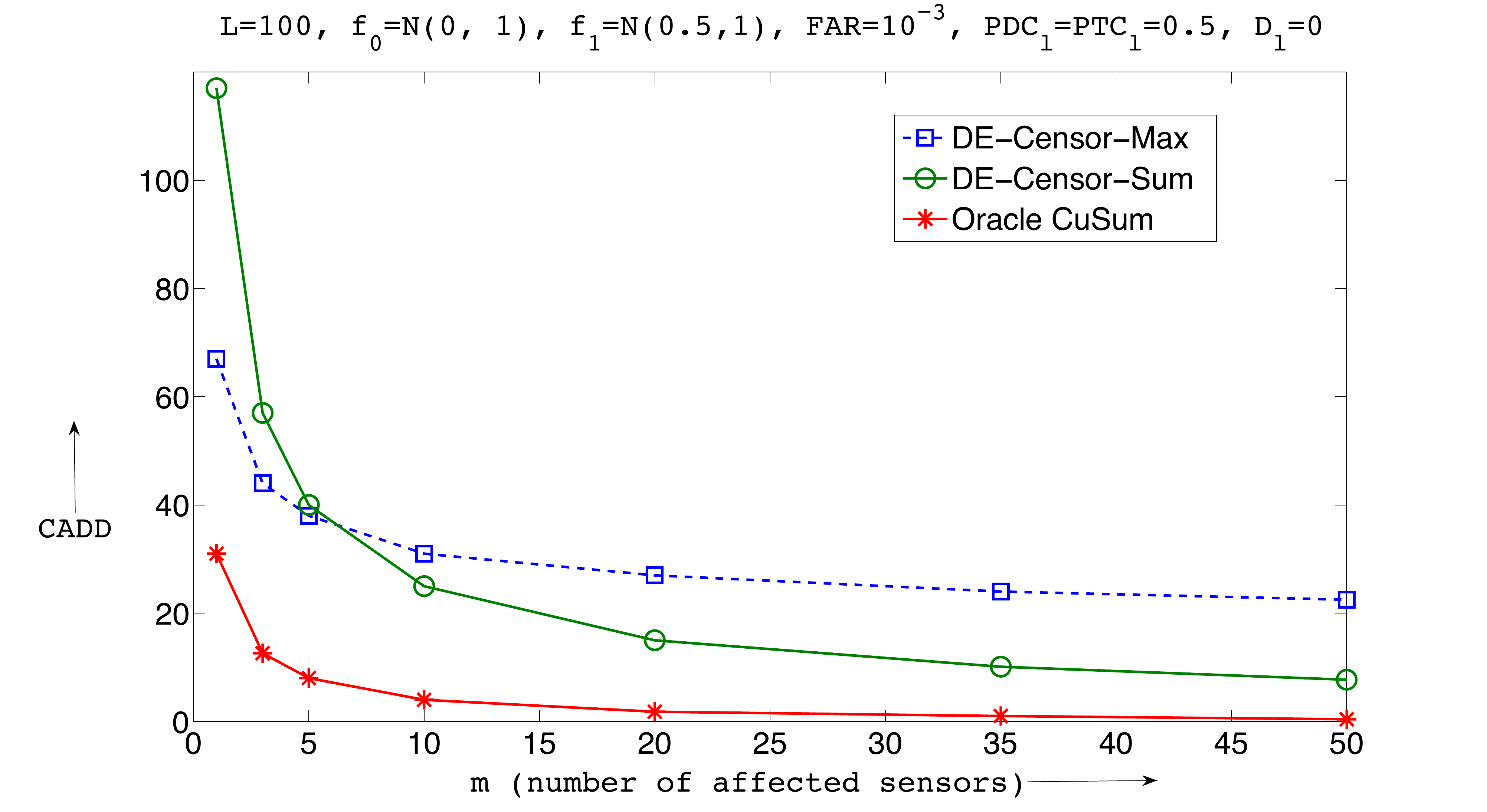}
\caption{{Comparison of DE-Censor-Sum algorithm, the DE-Censor-Max algorithm and the
Oracle CuSum algorithm as a function of the number of outlying streams.}}
\label{fig:MAX_SUM_Comp}
\end{figure}
In the figure we see that the DE-Censor-Max scheme outperforms the DE-Censor-Sum scheme when
the number of affected streams is small. This is because the former is optimal when the number of affected stream is
exactly one. However, when the number of affected streams is large, the DE-Censor-Sum algorithm outperforms the
DE-Censor-Max algorithm. We note that this is consistent with the observations made in \cite{mei-biometrica-2010}
regarding the comparison between the MAX and SUM algorithms. 

In Fig.~\ref{fig:MAX_Frac_CentCuSum} we compare the $\CADD$ vs $\FAR$ performance of the DE-Censor-Sum algorithm with the fractional sampling scheme
for $L=10$, $f_{0,\ell}=f_0=\mathcal{N}(0,1)$, $\forall \ell$, $f_{1,\ell}=f_1=\mathcal{N}(0.2,1)$, $\forall \ell$, and
for the $\{\PDC_\ell\}$ and $\{\PTC_\ell\}$ constraints of $\beta_\ell=\sigma_\ell=0.5$, $\forall \ell$. We consider
the post-change scenario when $m=7$.
We restrict our numerical study to the comparison of the $\CADD$ performance. Similar comparison can
be obtained for the $\WADD$ as well.

In the fractional sampling scheme, the CuSum algorithm is used at each sensor,
and samples are skipped based on the outcome of a sequence of fair coin tosses,
independent of the observation process.
If an observation is taken at a sensor, the CuSum statistic is transmitted to the fusion center.
Thus, achieving the constraints on the $\{\PDC_\ell\}$ and $\{\PTC_\ell\}$, $\forall \ell$. At the fusion center
a change is declared the first time
the sum of the CuSum statistics from all the sensors crosses a threshold. At the fusion center,
in the absence of any transmission from a sensor, its CuSum statistics from the last time instant is
used to compute the sum.
For the DE-Censor-Sum algorithm, we set $D_\ell=0$, $\{h_\ell=h = 10\}$, $\forall \ell$,
and use the approximation \eqref{eq:PDCApprox3} to select $\mu_\ell$.
This ensures that the $\{\PDC_\ell\}$ and $\{\PTC_\ell\}$ constraints are satisfied for
the DE-Censor-Sum algorithm. 
In the figure we see that the DE-Censor-Sum algorithm provides a significant
gain in performance as compared to the approach of fractional sampling.

\begin{figure}[htb]
\center
\includegraphics[width=9cm, height=5cm]{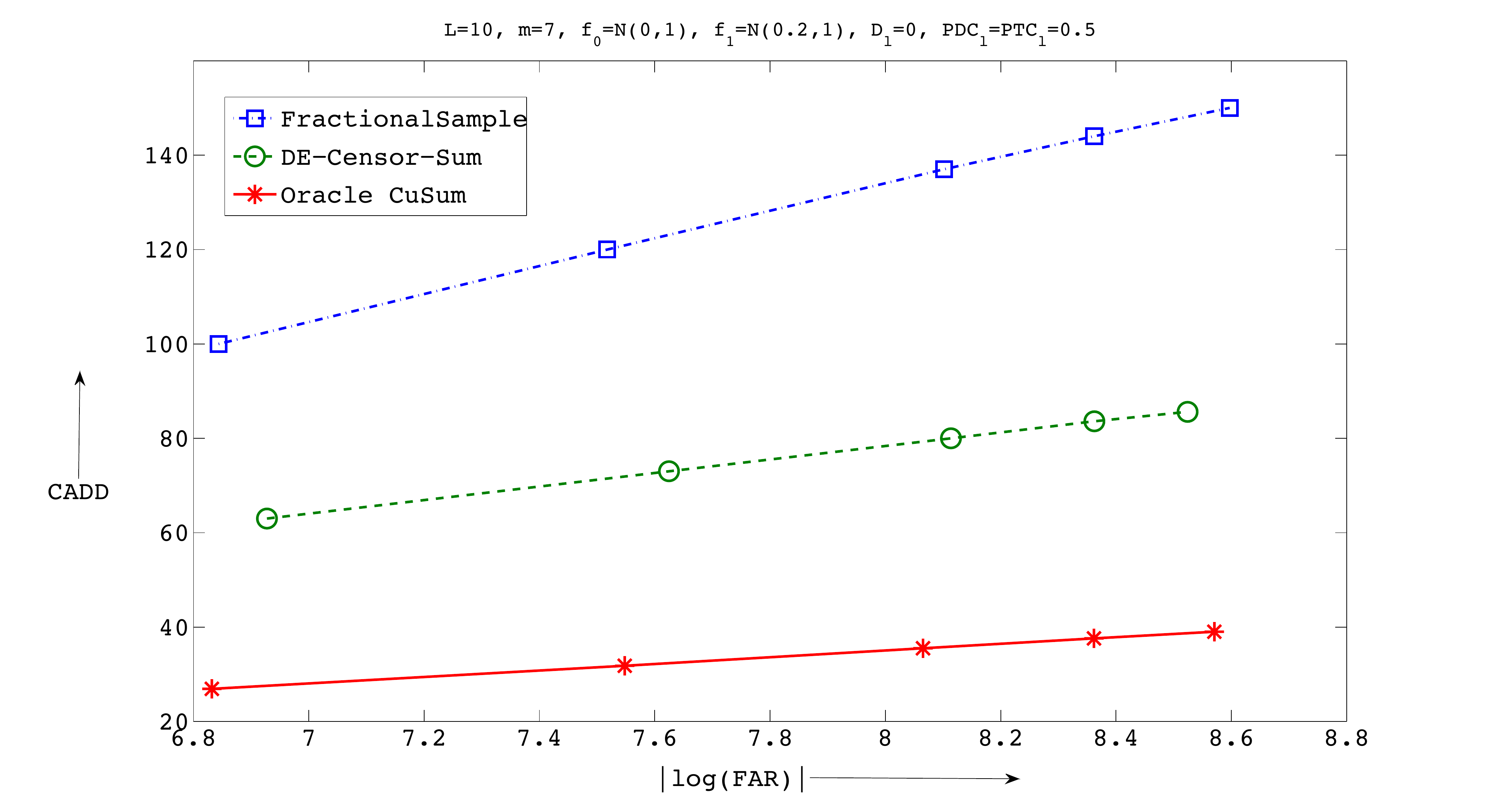}
\caption{{Comparison of the DE-Censor-Sum algorithm with the fractional sampling scheme.}}
\label{fig:MAX_Frac_CentCuSum}
\end{figure}

\section{Conclusions}
In this paper we proposed two data-efficient algorithms, 
the DE-Censor-Max algorithm and the DE-Censor-Sum algorithm, for quickest 
outlying sequence detection inn sensor networks when
the subset of affected sensors post-change is unknown to the decision maker. 
We provided a detailed performance analysis of these algorithms and compared 
their performance as a function of the number of affected sensors. 
We showed that the DE-Censor-Max algorithm is asymptotically optimal for 
the proposed formulations if exactly one sensor is affected post-change. 
Also, if the threshold can be appropriately selected, then the DE-Censor-Sum 
algorithm is asymptotically optimal, for every possible post-change scenario. 
We also showed via simulations that the DE-Censor-Max algorithm performs better 
if the number of affected sensors is small. Moreoever, our algorithms for 
observation control provide significant benefit over the approach of fractional sampling. 

\footnotesize
\bibliographystyle{ieeetr}



\bibliography{QCD_verSubmitted}
\end{document}